\documentclass[12pt]{amsart}
\usepackage{amsmath,amssymb,latexsym, amsthm, amscd, mathrsfs}
\usepackage{amsfonts,amssymb}
\usepackage{amsmath,amscd}
\usepackage{soul}
\usepackage[all]{xy}
\usepackage{color}
\theoremstyle{definition}
\theoremstyle{plain}
\newtheorem{prop}[subsection]{Proposition}
\newtheorem{thm}[subsection]{Theorem}
\newtheorem{lem}[subsection]{Lemma}
\newtheorem{cor}[subsection]{Corollary}

\newcommand{\mbf}{\mathbf}
\newcommand{\msf}{\mathsf}
\newcommand{\mbb}{\mathbb}
\newcommand{\mscr}{\mathscr}

\newcommand{\mcal}{\mathcal}
\newcommand{\mrk}{\mathfrak}

\newcommand{\U}{\mbf U}

\title[An identification] {An identification of Lusztig's modified forms of  quantum algebras and their analogues}

\author{ Zhaobing Fan and Yiqiang Li}
\address{State University of New York  at Buffalo \\Buffalo, NY 14260}
\email{
zhaobing@buffalo.edu (Z.Fan),
yiqiang@buffalo.edu (Y.Li)
}
\date{\today}
\keywords{Two/multi-parameter quantum algebras, multi-parameter quantum superalgebras, modified quantum algebras,  weight categories, canonical bases}
\subjclass[2010]{17B37}

\begin{document}
\begin{abstract}
We introduce a multi-parameter twisted Hopf algebra associated with a root datum and show that its modified form is isomorphic to
 Lusztig's modified quantum algebra under certain restrictions on the parameters.
 By taking various specializations of the parameters,
 we obtain similar results for two/multi-parameter quantum algebras and  Kang-Kashiwara-Oh's multi-parameter quantum superalgebras.
 As a consequence, the categories of weight modules  of these algebras are identical.
\end{abstract}

\maketitle

\section{Introduction}\label{Introduction}

Quantum algebras were  introduced by Drinfeld and Jimbo in the 1980s  involving  an auxiliary parameter $v$.
When $v$ is specialized to $1$, these algebras  become the universal enveloping algebras attached to   Kac-Moody Lie algebras.
Their representation theory has seen remarkable developments, such as the canonical basis theory.

In the  work ~\cite{FL12}, we investigate the effect of the extra ingredient, the  Tate twist, in
Lusztig's geometric construction ~\cite{Lusztig93} of negative halves of quantum algebras and their canonical bases.
Our desire to initiate such an investigation is inspired by the recent advancement in supercategorification by Hill-Wang ~\cite{HW12},  where the negative half of an analogue of quantum algebras with an extra parameter $\pi$ subject to $\pi^2=1$ are categorified.
It turns out that  the Tate twist can be naturally served as the second parameter $t$ for the two-parameter analogues of quantum algebras and from which we develop a theory of canonical basis for the negative halves of the  two-parameter analogues. This theory is essentially parallel to  that in the ordinary quantum algebras due to the observation that  the two types of algebras  differ by a two-cycle deformation depending only  on the second parameter.
To this end, we are motivated to  see what one can deduce from another geometric construction, due to
 Beilinson-Lusztig-MacPherson (\cite{BLM90}),  of the whole quantum algebras attached to $\mathfrak{gl}(n)$ by adding the Tate twist.

The first result we obtain is not surprising in that one can still  give a realization of the $q$-Schur quotient of two-parameter analogues of quantum $\mathfrak{gl}(n)$ by using mixed perverse sheaves on partial flag varieties.
What is non-trivial  from the construction is that  the $q$-Schur quotients of both types of algebras  are  identical in the geometric setting, unlike  the negative-half situation  in ~\cite{FL12}.  The result provides with us  a strong   evidence in support of  the speculation that the modified forms of quantum algebras  defined by Lusztig (\cite{Lusztig92, Lusztig93}) and their two-parameter analogues should be identical since the modified forms are in the projective limit of the $q$-Schur quotients.

As one of the main results, we show that Lusztig's modified forms and their two-parameter analogues are isomorphic   indeed.
Moreover, this isomorphism  can be established over the ring $\mbb Z[v^{\pm 1}, t^{\pm1}]$ of Laurent polynomials in two variables.

Via this isomorphism, many longstanding problems  in the study of two-parameter quantum algebras are solved by using known results from
 ordinary quantum algebras.
For example, the modified forms for two-parameter analogues admit  canonical bases.
Furthermore, the decomposition multiplicities  of irreducible modules in a Verma module for the two-parameter analogues are governed
by the Kazhdan-Lusztig polynomials (\cite{KL79}) if the parameters are taken to be generic.
%Note that the latter problem has been untouched in literature ever since the creation of two-parameter quantum algebras due to the lack of the appropriate geometry.

Our approach can be applied as well to the multi-parameter  and multi-parameter super analogues in  ~\cite{HPR10, KKO13}. We obtain similar results with certain restrictions on the parameters.

Instead of proving the results case by case,  we present a uniform proof. We define   a new multi-parameter twisted Hopf algebra  associated to a root datum  (see Section ~\ref{newU}).
The new algebra can be thought of as  a universal/unified algebra in the sense that
the ordinary quantum algebra and its generalizations mentioned above are
 specializations of this algebra under certain assumptions on the parameters.
 We then show that the modified form of this new algebra is isomorphic to that of ordinary quantum algebras  by twisting the Chevalley generators appropriately with certain restrictions on the parameters.
 The results for the above mentioned generalizations follow simultaneously  from this result by specialization of parameters.
 Notice that the crucial twists on the generators  go back to a geometric shift in  ~\cite{Li10} and the total order on the index set $I$  for the super setting first appears in ~\cite{CFLW}.

We refer to ~\cite{KKO13} for related results on weight categories for the  (multi-parameter) super analogues with generic parameters.
The analogous results for quantum super algebras are also obtained  in ~\cite{FL13} for the rank-one case and ~\cite{CFLW} for the general case together with Clark and Wang. See also ~\cite{CHW1, CHW2, CHW3} for the related theory of canonical bases of quantum superalgebras.

It is likely that one can obtain similar results for small quantum groups, which we hope to return to  in the future.

{\bf Acknowledgment.}
Y. Li is  supported partially by the NSF grant DMS-1160351.

%\setcounter{tocdepth}{1}
%\tableofcontents

\section{Quantum algebra}

\subsection{}

Let us fix  a Cartan datum $(I,\cdot)$ and an $X$-regular root datum $(Y, X, \langle,\rangle)$ of type $(I,\cdot)$.
We will use freely the following notations in this paper.
\begin{itemize}
\item $\{ \alpha_i|i\in I\}$ and  $\{\check \alpha_i| i\in I\}$ --- the set of simple roots and coroots, repectively.

\item $\{ \check{\omega}_i:  \ \forall i\in I\}$  --- the set of fundamental coweights.

\item $d_i=i\cdot i/2.$

\item $a_{ij} =2\frac{i\cdot j}{i\cdot i} =\langle \check{\alpha}_i,\alpha_j \rangle$, $\forall i, j\in I$.

\item $r=1-a_{ij}$, $\forall i\neq j\in I$.

\item $\lambda_i=\langle \check{\alpha}_i,\lambda\rangle$ and  $\lambda(i)=\langle\check{\omega}_i, \lambda\rangle$, for any $\lambda \in X$.
\end{itemize}

\subsection{}

Let $v$ be a generic parameter.  We set
$$\begin{array}{lll}
  [n]_v=\frac{v^n-v^{-n}}{v-v^{-1}},& [n]_v^!=\prod_{p=1}^n[p]_v, &{\rm and}\  \begin{bmatrix}n\\p\end{bmatrix}_v=\frac{[n]_v^!}{[p]_v^![n-p]_v^!},\quad \forall 0\leq p\leq n.
\end{array}$$
We further set
 \begin{equation}\label{eq23}
   v_i=v^{d_i}, \quad \forall i\in I.
 \end{equation}
When we write a notation like $[n]_{v_i}$, it stands for the polynomial obtained by substituting $v$ by $v_i$.

\subsection{}

The {\it quantum algebra} $\U$ associated to the root datum $(Y, X, \langle,\rangle)$ is a unital  associative $\mathbb{Q}(v)$-algebra  generated by the symbols $E_i, F_i$, and $ K_i^{\pm 1},$ $\forall i\in I$ and subject to the following relations.
\begin{align*}
&K_iK_j=K_jK_i,\ \  K_iK_i^{-1}=1=K_i^{-1}K_i. \vspace{6pt}\tag{$\U a$}\\
&K_iE_jK^{-1}_i=v_i^{a_{ij}}E_j,\ \ K_iF_jK^{-1}_i=v_i^{-a_{ij}}F_j. \vspace{6pt}\tag{$\U b$}\\
&E_iF_j-F_j E_i=\delta_{ij}\frac{K_i-K^{-1}_i}{v_i-v^{-1}_i}. \vspace{6pt} \tag{$\U c$}\\
&\textstyle \sum_{l=0}^{r} (-1)^lE_i^{(r-l)}E_j E_i^{(l)}=0,\quad %{\rm if}\ i\not =j,
\textstyle \sum_{l=0}^{r}  (-1)^lF_i^{(l)}F_j F_i^{(r-l)}=0,\quad {\rm if}\ i\not =j,\vspace{6pt}\tag{$\U d$}
\end{align*}
where
 $E^{(l)}_i=\frac{E_i^l}{[l]^!_{v_i}}$ and
 $F^{(l)}_i=\frac{F_i^l}{[l]^!_{v_i}}$.

  A $\U$-module $M$ is called a {\it weight module} if  $M=\oplus_{\lambda \in X}M_{\lambda}$, where
\begin{equation}\label{eq4}
M_{\lambda}=\{m\in M\ |\ K_i\cdot m=
v_i^{\lambda_i}m,\quad  \forall i\in I\}.
\end{equation}
Let $\mathbf{C}$ be the category of weight modules of $\U$.

\subsection{}

The {\it modified quantum algebra} $\dot{\U}$ of Lusztig associated to the root datum $(Y, X, \langle,\rangle)$
is defined to be  an associative  $\mbb{Q}(v)$-algebra  without unit,
 generated by the symbols $1'_{\lambda}, E'_{\lambda,\lambda-\alpha_i}$ and $F'_{\lambda,\lambda+\alpha_i}$, $\forall \lambda\in X$
 and subject to the following relations.
\begin{align*}
&1'_{\lambda} 1'_{\lambda'} =\delta_{\lambda, \lambda'} 1'_{\lambda}.  \vspace{6pt}\tag{$\dot{\U}a$} \\
&   E'_{\lambda+\alpha_i, \lambda }  1'_{\lambda'} =  \delta_{\lambda, \lambda'} E'_{\lambda+\alpha_i, \lambda}, \quad
1'_{\lambda'}   E'_{\lambda, \lambda-\alpha_i} = \delta_{\lambda', \lambda} E'_{\lambda, \lambda-\alpha_i}, \vspace{6pt}\tag{$\dot{\U}b$}\\
& F'_{\lambda-\alpha_i, \lambda}  1'_{\lambda'}=  \delta_{\lambda, \lambda'} F'_{\lambda-\alpha_i, \lambda}, \quad
1'_{\lambda'}  F'_{\lambda, \lambda+\alpha_i}=  \delta_{\lambda', \lambda} F'_{\lambda, \lambda+\alpha_i}. \vspace{6pt}\\
& E'_{\lambda,\lambda-\alpha_i} F'_{\lambda-\alpha_i,\lambda-\alpha_i+\alpha_j}-F'_{\lambda,\lambda+\alpha_j}  E'_{\lambda+\alpha_j,\lambda+\alpha_j-\alpha_i}=\delta_{ij}
[\lambda_i]_{v_i}1'_\lambda. \vspace{6pt}\tag{$\dot{\U}c$}\\
 & \textstyle \sum_{l=0}^{r}(-1)^lE'^{(r-l)}_{\lambda+r\alpha_i+\alpha_j, \lambda+l\alpha_i+\alpha_j}E'_{\lambda+l\alpha_i+\alpha_j, \lambda+l\alpha_i}
  E'^{(l)}_{\lambda+l\alpha_i,\lambda}=0,\vspace{6pt}\tag{$\dot{\U}d$}\\
  &\textstyle \sum_{l=0}^{r}(-1)^lF'^{(l)}_{\lambda, \lambda+l\alpha_i}F'_{\lambda+l\alpha_i, \lambda+l\alpha_i+\alpha_j}
  F'^{(r-l)}_{\lambda+l\alpha_i+\alpha_j,\lambda +r\alpha_i+\alpha_j}=0,\quad \forall i\neq j,
\end{align*}
where
   $E'^{(l)}_{\lambda+l\alpha_i,\lambda}=\frac{1}{[l]^!_{v_i}}
E'_{\lambda+l\alpha_i,\lambda+(l-1)\alpha_i}\cdots
E'_{\lambda+\alpha_i,\lambda}$
 and $F'^{(l)}_{\lambda,\lambda+l\alpha_i}=\frac{1}{[l]^!_{v_i}}
F'_{\lambda,\lambda+\alpha_i}\cdots F'_{\lambda+(l-1)\alpha_i,\lambda+l\alpha_i}
$.
Set $$\mbf{A}=\mbb{Z}[v,v^{-1}].$$
Let ${}_{\mbf A}\!\dot{\U}$ be the $\mbf{A}$-subalgebra of $\dot{\U}$ generated by
 $E'^{(l)}_{\lambda+l\alpha_i,\lambda}$, $F'^{(l)}_{\lambda,\lambda+l\alpha_i}$ and $1'_{\lambda}$
 for all $\lambda\in X$, $i\in I$ and $l\in \mbb N$.
Let $\dot{\mbf{C}}$  be the category of unital $\dot{\U}$-modules.
% in the sense of Lusztig ~\cite[23.1.4]{Lusztig93}.
By ~\cite[23.1.4]{Lusztig93},
 we have an isomorphism of categories
 $$\mathbf{C}\simeq \dot{\mbf{C}}.$$

\section{The twisted Hopf algebra $\mscr{U}_{\mbf{q,s,t}}$}

\label{newU}
\subsection{}

Let $q_{i}, s_{ij}, t_{ij}\ \forall i, j\in I $ be indeterminates over $\mbb{Q}$.
For short, we denote $\mbf{q}=(q_i)_{i\in I}, \mathbf{s}=(s_{ij})_{i,j\in I},$ and $\mathbf{t}=(t_{ij})_{i,j\in I}$.
Let $\mbb{K}=\mbb{Q}(s_{ij}, t_{ij}, q_{i}^{1/d_i})_{i,j\in I}$ be the field of rational functions in the indicated variables,  where  the parameter $q_i^{1/d}$ is a formal parameter such that its $d$-th power is equal to $q_i$. 
 The algebra  $ \mscr{U}\equiv \mscr{U}_{\mbf{q, s, t}}$ associated to the root datum $(Y, X, \langle, \rangle)$ is an associative $\mbb{K}$-algebra with 1 generated by the symbols $E_i, F_i, K_i^{\pm 1}$ and $ K_i'^{\pm 1},$ $\forall i\in I$, and subject to the following  relations.
\begin{align*}
  & K_iK_j=K_jK_i,\ \ K'_iK'_j=K'_jK'_i,\ \  K_iK'_j=K'_jK_i,\vspace{6pt}\tag{$\mscr{U} a$}\\
     & K_iK_i^{-1}=K_i^{-1}K_i=1,\ \ K'_iK'^{-1}_i=K'^{-1}_iK'_i=1. \vspace{6pt}\\
  & K_iE_jK^{-1}_i=s^{-1}_{ij}t^{-1}_{ij}q_i^{a_{ij}}E_j,\ \ K'_iE_jK'^{-1}_i=s^{-1}_{ij}t^{-1}_{ij}q_i^{-a_{ij}}E_j, \vspace{6pt}\tag{$\mscr{U} b$}\\
     & K_iF_jK^{-1}_i=s_{ij}t_{ij}q_i^{-a_{ij}}F_j,\ \ K'_iF_jK'^{-1}_i=s_{ij}t_{ij}q_i^{a_{ij}}F_j. \vspace{6pt}\\
  & E_iF_j-s_{ij}t_{ji}F_j E_i=\delta_{ij}\frac{K_i-K'_i}{q_i-q^{-1}_i}. \vspace{6pt}\tag{$\mscr{U} c$}\\
  &\textstyle  \sum_{l=0}^{r}(-1)^ls_{ji}^ls_{ij}^{-l}
  E_i^{\langle r-l\rangle}E_j E_i^{\langle l\rangle}=0,\quad %{\rm if}\ i\not =j, \vspace{6pt}\\
   \textstyle \sum_{l=0}^{r}(-1)^lt_{ji}^lt_{ij}^{-l}
  F_i^{\langle l\rangle }F_j F_i^{\langle r-l\rangle }=0,\quad {\rm if}\ i\not =j\tag{$\mscr{U} d$},
\end{align*}
where $E^{\langle l\rangle }_i=\frac{E_i^l}{[l]^!_{q_i}}$
and $F^{\langle l\rangle}_i=\frac{F_i^l}{[l]^!_{q_i}}$.
We associate to $\mscr U$ a $\mbb Z[I]$-grading by setting
$$deg(E_i)=i,\quad deg(F_i)=-i\quad {\rm and }\quad deg(K_i)=0=deg(K_i'),\quad \forall i\in I,$$
For any homogenous element $x\in \mscr U$, let $|x|$ be its degree.
For any $\mu, \nu\in \mbb Z[I]$,
let
$$\textstyle t_{\mu, \nu}=\prod_{i, j\in I}t_{ij}^{\mu_i \nu_j} \quad {\rm and} \quad
s_{\mu, \nu}=\prod_{i, j\in I}s_{ij}^{\mu_i \nu_j}.
%\quad q^{a_{\mu, \nu}}=\prod_{i, j\in I}q_i^{\mu_i \nu_j a_{ij}}.
$$
We twist the algebra structures on $\mscr U\otimes \mscr U$ and $\mscr U\otimes \mscr U\otimes \mscr U$ by
\begin{equation}\label{eq44}
\begin{split}
&  (x_1\otimes x_2)(y_1\otimes y_2)=t_{|x_2|, |y_1|}s_{|y_1|, |x_2|}x_1y_1 \otimes x_2y_2,\quad \mbox{and}\\
&  (x_1\otimes x_2\otimes x_3)(y_1\otimes y_2\otimes y_3)=t_{|x_2|+|x_3|, |y_1|}s_{|y_1|, |x_2|+|x_3|}
t_{|x_3|, |y_2|}s_{|y_2|, |x_3|}
x_1y_1 \otimes x_2y_2 \otimes x_3y_3
\end{split}
\end{equation}
for any homogeneous elements $x_1, x_2, x_3, y_1, y_2$, and $y_3 \in \mscr U.$
\begin{lem}
 If $q_i^{a_{ij}}=q_{j}^{a_{ji}}, \ \forall i, j\in I$,
 then  $\mscr U$ admits a bialgebra structure when equipped with the following comultiplication $\Delta$ and  counit $\varepsilon$.
  \begin{equation}\label{eq45}
  \begin{split}
   &\Delta(E_i)=E_i\otimes 1+K_i\otimes E_i,\quad \Delta(K_i^{\pm 1})=K_i^{\pm 1} \otimes K_i^{\pm 1}, \\
   &\Delta(F_i)=1 \otimes F_i+F_i\otimes K'_i, \quad  \Delta(K'^{\pm 1}_i)=K'^{\pm 1}_i \otimes K'^{\pm 1}_i,\vspace{4pt}\\
   &\varepsilon(K_i^{\pm 1})=\varepsilon(K'^{\pm 1}_i)=1,\quad \hspace{20pt}\varepsilon(E_i)=\varepsilon(F_i)=0,\quad \forall i\in I.\\
  \end{split}
  \end{equation}
\end{lem}
\begin{proof} 
We shall show that the comultiplication $\Delta$ is well defined, i.e., the evaluations of $\Delta$ at both sides of the defining relations of $\mscr U$ coincide.
It is straightforward  to verify that this is the case for the relations in
$(\mscr Ua), (\mscr Ub)$ and $(\mscr Uc)$.
We are  left  to check that the evaluations of $\Delta$ at the left-hand sides in $(\mscr Ud)$ are zero.
%The proof goes in a similar way as that of Lemma 3.14 in \cite{Lusztig93}, which we sketch below.
Set
  $$R_{ij}=\sum_{l=0}^r(-1)^ls_{ji}^ls_{ij}^{-l}\begin{bmatrix}
    r\\l
  \end{bmatrix}_{q_i}E_i^{r-l}E_jE_i^l,\quad \forall i\neq j\in I.$$
  It is convenient to calculate $\Delta(R_{ij})$ in   the  associative algebra
  ${}'\!\mscr U^+$,  generated by $E_i, K_i^{\pm 1},\ \forall i\in I$ and subject to
  the relations ($\mscr{U} a$) and ($\mscr{U} b$).
(Our previous arguments show that the assignments for $\Delta$ on $E_i$ and $K_i$ define an algebra homomorphism on ${}'\!\mscr U^+$.)

  By (\ref{eq44}) and the  relation $(\mscr Ub)$, we have
\begin{equation}
  \label{eqKE}
  (K_i\otimes E_i)(E_i\otimes 1)=q_i^2(E_i\otimes 1)(K_i\otimes E_i).
\end{equation}
 By Section 1.3.5 in \cite{Lusztig93} and (\ref{eqKE}), we have
\begin{equation}
\label{eq1.3.5}
\Delta(E_i^n)=(\Delta(E_i))^n=\sum_{l=0}^n q_i^{l(n-l)}\begin{bmatrix}
  n\\l
\end{bmatrix}_{q_i}(E_i^n\otimes 1)(K_i^{n-l}\otimes E_i^{n-l}).
\end{equation}
  Recall from ~\cite[Section 1.3.4]{Lusztig93} that  we have the following identity
  \begin{equation}\label{eq1.3.4}
    \sum_{l=0}^n(-1)^lq_i^{l(1-n)}\begin{bmatrix}
    n\\l
  \end{bmatrix}_{q_i}=0,\quad \forall n\geq 1.
\end{equation}
%For any $x, y\in {}'\!\mscr U^+$ with $xy=q_i^2yx$, by Sections 1.3.5 in \cite{Lusztig93}, we have
%\begin{equation}
%\label{eq1.3.5}
%(x+y)^r=\sum_{l=0}^r q_i^{l(r-l)}\begin{bmatrix}
%  r\\l
%\end{bmatrix}_{q_i}y^lx^{r-l}.
%\end{equation}
With the help of  (\ref{eq1.3.5}) and (\ref{eq1.3.4}), one can directly check that
\begin{equation} \label{eq48}
\Delta(R_{ij})=R_{ij}\otimes 1+K_i^rK_j\otimes R_{ij} \quad \mbox{in ${}'\!\mscr U^+$}.
\end{equation}
%We note that a detail proof of (\ref{eq48}) for multi-parameter quantum algebras is given in  ~\cite[Appendix]{HPR10}.
Let $\mscr U^+$ be the subalgebra of $\mscr U$ generated by $E_i, K_i^{\pm 1},\ \forall i\in I$.
It is the same as the quotient algebra of ${}'\!\mscr U^+$ by the two-sided ideal $\mscr I$ generated by $R_{ij}, \forall i\neq j\in I$.
By (\ref{eq48}), we have immediately that  $\Delta(R_{ij})=0$ in $\mscr U^+$.
The second relation in ($\mscr{U} d$) can be shown similarly.
This shows that $\Delta$ is well defined.
 The coassociativity of $\Delta$ and the counit can be verified directly.
   \end{proof}

We define a new multiplication on $\mscr U$ by
$$x * y= s_{|y|,|x|}t_{|x|, |y|} yx, \ \forall x, y\in \mscr U\ {\rm homogeneous}.$$
Denote by $\mscr U^{*op}$ the resulting  algebra.
\begin{prop}
 Assume that  $q_i^{a_{ij}}=q_{j}^{a_{ji}}$ for any  $i, j\in I$.
 %\begin{equation}\label{eq46}
%   q_i^{a_{ij}}=q_{j}^{a_{ji}},  %\quad s_{ij}s_{ji}t_{ij}t_{ji}=1
%   \quad {\rm and}\quad s_{ij}t_{ji}=s_{ji}t_{ij}=s^{-1}_{ij}t^{-1}_{ji},
%   \ \forall i, j\in I.
% \end{equation}
 The algebra $\mscr U$ becomes a twisted Hopf algebra when it is equipped with the comultiplication $\Delta$ and  the counit $\varepsilon$ in (\ref{eq45}) together with   the antipode $S: \mscr U\rightarrow \mscr U^{*op}$ defined as follows.
  $$\begin{array}{llll}
 %  &\Delta(K_i^{\pm 1})=K_i^{\pm 1} \otimes K_i^{\pm 1},&\Delta(K'^{\pm 1}_i)=K'^{\pm 1}_i \otimes K'^{\pm 1}_i, &\vspace{4pt}\\
%   &\Delta(E_i)=E_i\otimes 1+K_i\otimes E_i,& \Delta(F_i)=1 \otimes F_i+F_i\otimes K'_i, & \vspace{4pt}\\
 %\varepsilon(E_i)=0,& \varepsilon(F_i)=0,&\varepsilon(K_i^{\pm 1})=1, &\varepsilon(K'^{\pm 1}_i)=1,\vspace{4pt}\\
  S(E_i)=-K_i^{-1}E_i,&S(F_i)=-F_iK'^{-1}_i,& S(K_i^{\pm 1})=K_i^{\mp 1},
  & S(K'^{\pm 1}_i)=K'^{\mp 1}_i.
  \end{array}$$
\end{prop}
\begin{proof}
  We shall show that $S: \mscr U\rightarrow \mscr U^{*op}$ is an algebra homomorphism.
  It is easy to check that $S$ is compatible with the relations $(\mscr U a)$ and $(\mscr Ub)$.
  For the relation $(\mscr Uc)$,  we have
  \begin{equation*}
    \begin{split}
    &S(E_iF_j-s_{ij}t_{ji}F_jE_i)=s_{ji}^{-1}t_{ij}^{-1}F_jK'^{-1}_jK_i^{-1}E_i-K_i^{-1}E_iF_jK'^{-1}_j\\
    &=s_{ji}^{-1}s_{jj}t_{jj}q_j^{a_{ji}}s_{ij}q_i^{-a_{ij}}K'^{-1}_j K_i^{-1}F_jE_i-
    s_{jj}t_{jj}q_j^{a_{ji}}s_{ji}^{-1}t_{ji}^{-1}q_j^{-a_{ji}}K'^{-1}_j K_i^{-1}E_i F_j\\
    &=s_{jj}t_{jj}s_{ji}^{-1}t_{ji}^{-1}K'^{-1}_j K_i^{-1}(s_{ij}t_{ji}F_jE_i-E_iF_j)
    =\delta_{ij}\frac{S(K_i)-S(K'_i)}{q_i-q_i^{-1}}.
    \end{split}
  \end{equation*}
  This shows that $S$ is compatible with the relation $(\mscr Uc)$.
   Since $S(E_i^l)=(-1)^lq_i^{l(l-1)}K_i^{-l}E_i^l$, we have
  \begin{equation}\label{eq47}
  \textstyle  S( \sum_{l=0}^{r}(-1)^ls_{ji}^ls_{ij}^{-l}
  E_i^{\langle r-l\rangle}E_j E_i^{\langle l\rangle})
  =N \sum_{l=0}^{r}(-1)^ls_{ij}^ls_{ji}^{-l}K_i^{-r}K_j
  E_i^{\langle l\rangle}E_j E_i^{\langle r-l\rangle},
  \end{equation}
 where $N$ is a scalar  independent of $l$. By changing the index, (\ref{eq47}) is zero.
  The second relation in $(\mscr Ud)$ can be checked similarly.
  This shows that $S$ is a well-defined endomorphism of $\mscr U$.
  It is now straightforward to verify that $\mscr U$ together with the  associated data satisfies the Hopf algebra axioms.
  This finishes the proof.
\end{proof}

 A $\mscr{U}$-module $M$ is called a {\it weight module} if  it admits a decomposition $M=\oplus_{\lambda \in X}M_{\lambda}$ of vector spaces such that
$$\textstyle M_{\lambda}=\{m\in M\ |\ K_i\cdot m=
 \msf c_{i,\lambda} q_i^{\lambda_i}m,\ \ K'_i\cdot m=
 \msf c_{i,\lambda} q_i^{-\lambda_i}m,\quad  \forall i\in I\},$$
where
\begin{equation}\label{eq43}
\textstyle \msf c_{i,\lambda}= \prod_{j\in I}(s_{ij}t_{ij})^{-\lambda(j)}.
\end{equation}
Let $\mscr{C}$ be the category of weight modules of $\mscr{U}$.

\subsection{}

The  $modi\! f\! ied$  $algebra$ $\dot{\mscr{U}}$ attached to $\mscr U$
is defined to be  an associative  $\mbb{K}$-algebra  without unit,
 generated by the symbols $1_{\lambda}, E_{\lambda,\lambda-\alpha_i}$ and $F_{\lambda,\lambda+\alpha_i}$, $\forall \lambda\in X$ and $i\in I$,
 and subject to the following relations.
\begin{align*}
&1_{\lambda} 1_{\lambda'} =\delta_{\lambda, \lambda'} 1_{\lambda}. \vspace{6pt}\tag{$\dot{\mscr U}a$}\\
&   E_{\lambda+\alpha_i, \lambda }  1_{\lambda'} =  \delta_{\lambda, \lambda'} E_{\lambda+\alpha_i, \lambda}, \quad
1_{\lambda'}   E_{\lambda, \lambda-\alpha_i} = \delta_{\lambda', \lambda} E_{\lambda, \lambda-\alpha_i},\vspace{6pt}\tag{$\dot{\mscr U}b$}\\
& F_{\lambda-\alpha_i, \lambda}  1_{\lambda'}=  \delta_{\lambda, \lambda'} F_{\lambda-\alpha_i, \lambda}, \quad
1_{\lambda'}  F_{\lambda, \lambda+\alpha_i}=  \delta_{\lambda', \lambda} F_{\lambda, \lambda+\alpha_i}.\vspace{6pt}\\
& \textstyle E_{\lambda,\lambda-\alpha_i} F_{\lambda-\alpha_i,\lambda-\alpha_i+\alpha_j}-s_{ij}t_{ji}F_{\lambda,\lambda+\alpha_j}  E_{\lambda+\alpha_j,\lambda+\alpha_j-\alpha_i}=\delta_{ij} \msf c_{i,\lambda}
[\lambda_i]_{q_i}1_\lambda.\vspace{6pt}\tag{$\dot{\mscr U}c$}\\
 &\textstyle \sum_{l=0}^{r}(-1)^ls_{ji}^ls_{ij}^{-l}E^{\langle r-l\rangle}_{\lambda+r\alpha_i+\alpha_j, \lambda+l\alpha_i+\alpha_j}E_{\lambda+l\alpha_i+\alpha_j, \lambda+l\alpha_i}
  E^{\langle l\rangle}_{\lambda+l\alpha_i,\lambda}=0,\vspace{6pt}
  \tag{$\dot{\mscr U}d$}\\
  &\textstyle \sum_{l=0}^{r}(-1)^lt_{ji}^lt_{ij}^{-l}F^{\langle l\rangle}_{\lambda, \lambda+l\alpha_i}F_{\lambda+l\alpha_i, \lambda+l\alpha_i+\alpha_j}
  F^{\langle r-l\rangle}_{\lambda+l\alpha_i+\alpha_j,\lambda +r\alpha_i+\alpha_j}=0,\quad \forall i\neq j,
\end{align*}
where
 $E^{\langle l\rangle}_{\lambda+l\alpha_i,\lambda}=\frac{1}{[l]^!_{q_i}}
E_{\lambda+l\alpha_i,\lambda+(l-1)\alpha_i}\cdots
E_{\lambda+\alpha_i,\lambda}$
 and
$F^{\langle l\rangle}_{\lambda,\lambda+l\alpha_i}=\frac{1}{[l]^!_{q_i}}
F_{\lambda,\lambda+\alpha_i}\cdots F_{\lambda+(l-1)\alpha_i,\lambda+l\alpha_i}$.
Let
$$\mathscr{A}=\mbb{Z}[s^{\pm 1}_{ij}, t^{\pm 1}_{ij}, q_{i}^{\pm 1/d_i}]_{i,j\in I}$$
be the ring of Laurent polynomials in variables $s_{ij}$, $t_{ij}$ and $q_i^{1/d_i}$.
Let ${}_{\mathscr A}\!\dot{\mathscr U}$ be the $\mathscr{A}$-subalgebra of
$\dot{\mathscr U}$ generated by
 $E^{\langle l\rangle }_{\lambda+l\alpha_i,\lambda}$, $F^{\langle l\rangle }_{\lambda,\lambda+l\alpha_i}$ and $1_{\lambda}$
 for all $\lambda\in X$, $i\in I$ and $l\in \mbb N$.

Let $\dot{\mathscr{C}}$  be the category of unital $\dot{\mathscr{U}}$-modules. Similar to the case of ordinary quantum algebras, we have

\begin{prop} \label{prop4}
  The categories $\dot{\mathscr{C}}$ and $\mathscr{C}$ are  isomorphic.
\end{prop}

\begin{proof}

 We define a functor $\eta: \mscr{C}\rightarrow \dot{\mscr{C}}$ as follows.
 Given a weight $\mscr{U}$-module $M=\oplus_{\lambda\in X}M_{\lambda}$,
 we define a $\dot{\mscr{U}}$-module on $M$ by
 \begin{equation}\label{eq1}
E_{\lambda'+\alpha_i, \lambda'}\cdot m=
\delta_{\lambda, \lambda'}E_i\cdot m,\quad
F_{\lambda'-\alpha_i, \lambda'}\cdot m=
\delta_{\lambda, \lambda'}F_i\cdot m\ {\rm and}\  1_{\lambda'}\cdot m= \delta_{\lambda, \lambda'}m,\quad \forall \ m\in M_{\lambda}.
 \end{equation}
It can be   checked that the $\dot{\mscr{U}}$-module structure on $M$ in (\ref{eq1}) is well-defined. Moreover,
 any morphism $f: M\to N$  in $\mscr{C}$ is again  a morphism in $\dot{\mscr C}$ if $M$ and $N$ are regarded as $\dot{\mscr{U}}$-modules.

We now define a functor $\eta': \dot{\mscr{C}}\rightarrow \mscr{C}.$
Given a unital $\dot{\mscr{U}}$-module $M$, let $M_{\lambda}=1_{\lambda}\cdot M$.
By the relation ($\dot{\U}a$), we have $M_{\lambda}\cap M_{\lambda'}=\{0\}$,
if $\lambda\neq \lambda'$.
Therefore,
we have a decomposition of vector space $M=\oplus_{\lambda\in X}M_{\lambda}$.
%\st{$\eta'(M)$ is the $\mscr{U}$-module $M$ defined by}
The underlying vector space of the $\mscr U$-module $\eta'(M)$ is $M$ and the module structure on $M$ is defined by
\begin{equation}\label{eq2}
\begin{split}
&E_i\cdot m=E_{\lambda+\alpha_i,\lambda}\cdot m,\quad F_i\cdot m=F_{\lambda-\alpha_i,\lambda}\cdot m
,\\
&\textstyle K_i\cdot m=
\msf c_{i,\lambda} q_i^{\lambda_i}m\  {\rm and}\  K'_i\cdot m=
\msf c_{i,\lambda} q_i^{-\lambda_i}m,\quad \forall \ m\in M_{\lambda},\ i\in I.
\end{split}
\end{equation}
One can  check that the $\mscr{U}$-module structure in (\ref{eq2}) is well-defined.
 Similarly, any morphism $f: M\to N$  in $\dot{\mscr C}$ is automatically a morphism in $\mscr{C}$ if $M$ and $N$ are regarded as $\mscr{U}$-modules.

 It is clear that the functors $\eta$ and $ \eta'$ are inverse to each other. This finishes the proof.
\end{proof}

\section{Main result}

In this section, we assume that
\begin{equation}\label{eq35}
  q_i=v_i, \quad \forall i\in I,
\end{equation}
where $v_i$ is defined in (\ref{eq23}).
Let $\psi$ be the assignment
\begin{eqnarray}\label{eq36}
  E'_{\lambda,\lambda-\alpha_i} \mapsto \msf{e}_{i,\lambda} E_{\lambda,\lambda-\alpha_i},\quad  F'_{\lambda-\alpha_i,\lambda} \mapsto \msf{f}_{i,\lambda} F_{\lambda-\alpha_i,\lambda},\quad 1'_{\lambda}\mapsto 1_{\lambda},\quad \forall \lambda\in X, i\in I,
\end{eqnarray}
where
\begin{equation}\label{eq37}
  \textstyle \msf{e}_{i,\lambda}=\prod_{j\in I}s_{ij}^{\lambda(j)}\quad  {\rm and}\quad
\msf{f}_{i,\lambda}=\prod_{j\in I}t_{ij}^{\lambda(j)}.
\end{equation}

\begin{thm}\label{thm8}
Under the assumption (\ref{eq35}), we have the following statements.
   \begin{itemize}
    \item[(a)] The assignment $\psi$ in (\ref{eq36}) defines an algebra isomorphism
    $\xymatrix { {\mbb K} \otimes_{\mbb{Q}(v)} \dot{\U} \ar[r]^-{\sim} &  \dot{\mscr{U}}}$.

    \item[(b)] The restriction of  $\psi$ to ${\mscr A} \otimes_{\mbf{A}} {}_{\mbf A}\!\dot{\U}$ induces an algebra isomorphism $ {\mscr A} \otimes_{\mbf{A}} {}_{\mbf A}\!\dot{\U} \simeq  {}_{\mscr A}\!\dot{\mscr{U}}$.

     \item[(c)]  The image of the canonical basis of $\dot{\U}$ under $\psi$  is a basis of $\dot{\mscr U}$ and a basis of ${}_{\mscr A}\!\dot{\mscr{U}}$.

    \item[(d)]  The category  $\mbf{C}$ of weight modules of $\mbf U$ over $\mbb K$ is isomorphic to  the category $ \mscr{C}$
                     of weight modules of $\mscr U$.
  \end{itemize}

\end{thm}

\begin{proof}
We need to show that the assignment $\psi$ in (\ref{eq36}) defines an algebra homomorphism.
%$$ \psi: {\mbb K} \otimes_{\mbb{Q}(v)} \dot{\U} \rightarrow  \dot{\mscr{U}}.$$
 By the definitions of $\msf e_{i,\lambda}$ and $\msf f_{i,\lambda}$ in (\ref{eq37}), we have
\begin{equation}\label{eq7}
\begin{split}
\msf e_{i, \lambda+\alpha_j}=\msf e_{i, \lambda} s_{ij},\quad \msf f_{i, \lambda+\alpha_j}=f_{i, \lambda}t_{ij},\quad \forall i, j\in I.
\end{split}
\end{equation}
By (\ref{eq7}), we have that
$\msf f_{j,\lambda+\alpha_j}\msf e_{i,\lambda+\alpha_j}=\msf e_{i,\lambda}\msf f_{j,\lambda-\alpha_i+\alpha_j}s_{ij}t_{ji}$ and
$\msf e_{i,\lambda}\msf f_{i,\lambda}=\msf c_{i,\lambda}^{-1}$. So
 \begin{equation}\label{eq6}
\begin{split}
& \psi( E'_{\lambda,\lambda-\alpha_i} F'_{\lambda-\alpha_i,\lambda-\alpha_i+\alpha_j}-F'_{\lambda,\lambda+\alpha_j}  E'_{\lambda+\alpha_j,\lambda+\alpha_j-\alpha_i}-\delta_{ij}[\lambda_i]_{v_i}1'_{\lambda})=0,\quad \forall i, j\in I.
\end{split}
\end{equation}
Under the assumption (\ref{eq35}), we have
$$\textstyle \psi(E'^{(l)}_{\lambda+l\alpha_i,\lambda})=
\prod_{k=1}^l\msf e_{i,\lambda+k\alpha_i}E^{\langle l\rangle}_{\lambda+l\alpha_i,\lambda}
=\msf e_{i,\lambda}^l s_{ii}^{l(l+1)/2}E^{\langle l\rangle}_{\lambda+l\alpha_i,\lambda}.$$
 So
\begin{equation*}
\begin{split}
&\psi(E'^{(r-l)}_{\lambda+r\alpha_i+\alpha_j, \lambda+l\alpha_i+\alpha_j}
E'_{\lambda+l\alpha_i+\alpha_j, \lambda+l\alpha_i}
  E'^{(l)}_{\lambda+l\alpha_i,\lambda})\vspace{6pt}
=N_1E^{\langle r-l\rangle}_{\lambda+r\alpha_i+\alpha_j, \lambda+l\alpha_i+\alpha_j}E_{\lambda+l\alpha_i+\alpha_j, \lambda+l\alpha_i}
  E^{\langle l\rangle}_{\lambda+l\alpha_i,\lambda}.
\end{split}
\end{equation*}
where $N_1=\msf e_{i,\lambda+l\alpha_i+\alpha_j}^{r-l}\msf e_{j,\lambda+l\alpha_i}
\msf e_{i,\lambda}^ls_{ii}^{l(l+1)/2+(r-l)(r-l+1)/2}$.
A direct calculation shows that $N_1=N_2s_{ij}^{-l}s_{ji}^l$,
where $N_2$ is a number independent of $l$. So
\begin{equation}\label{eq21}
\textstyle \psi(\sum_{l=0}^r(-1)^lE'^{(r-l)}_{\lambda+r\alpha_i+\alpha_j, \lambda+l\alpha_i+\alpha_j}
E'_{\lambda+l\alpha_i+\alpha_j, \lambda+l\alpha_i}
  E'^{(l)}_{\lambda+l\alpha_i,\lambda})=0
\end{equation}
Similarly, we have
\begin{equation}\label{eq22}
\textstyle \psi(\sum_{l=0}^r(-1)^lF'^{(l)}_{\lambda, \lambda+l\alpha_i}F'_{\lambda+l\alpha_i, \lambda+l\alpha_i+\alpha_j}
  F'^{(r-l)}_{\lambda+l\alpha_i+\alpha_j,\lambda +r\alpha_i+\alpha_j})=0.
\end{equation}

By (\ref{eq6}), (\ref{eq21}) and (\ref{eq22}), the assignment $\psi$ induces an algebra homomorphism
$\psi: {\mscr A} \otimes_{\mbf{A}} {}_{\mbf A}\!\dot{\U} \rightarrow  {}_{\mscr A}\!\dot{\mscr{U}}$.
 By a similar argument, the following map defines the inverse of $\psi$:
\begin{eqnarray*}
  E_{\lambda,\lambda-\alpha_i} \mapsto \msf e_{i,\lambda}^{-1} E'_{\lambda,\lambda-\alpha_i},\quad  F_{\lambda-\alpha_i,\lambda} \mapsto \msf f_{i,\lambda}^{-1} F'_{\lambda-\alpha_i,\lambda},\quad 1_{\lambda}\mapsto 1'_{\lambda}.
\end{eqnarray*}
This proves (a). Since $\psi$ maps $ {\mscr A} \otimes_{\mbf{A}} {}_{\mbf A}\!\dot{\U}$ onto $  {}_{\mscr A}\!\dot{\mscr{U}}$, the statement  (b) follows.
The statement  (c)  follows from  (a) and (b). The statement  (d) follows from  (a) and (b) and Proposition \ref{prop4}.
\end{proof}

We remark that (\ref{eq37}) is a multiplicative version of the numbers $e_{\lambda,\alpha_i}$ and $f_{\lambda, \alpha_i}$ in ~\cite[5.1]{Li10}.
In what follows, we shall  deduce similar results for various two/multiparameter and multiparameter super analogues by specializing the parameters.

\subsection{Two-parameter case}

 Let $\Omega=(${\scriptsize  $ \Omega$}$_{ij})_{i,j\in I}$ be a matrix satisfying that

\begin{itemize}
  \item[(a)] {\scriptsize  $ \Omega$}$_{ii} \in \mathbb{Z}_{>0}$, {\scriptsize  $ \Omega$}$_{ij}\in \mathbb{Z}_{\leq 0}$ for all $i\neq j \in I$;

  \item[(b)] $\frac{\Omega_{ij}+\Omega_{ji}}{\Omega_{ii}}\in \mathbb{Z}_{\leq 0}$ for all $i\neq j \in I$;

  \item[(c)] {\scriptsize  $ \Omega$}$_{ij}+${\scriptsize  $ \Omega$}$_{ji}=i\cdot j,$ for all $i, j\in I$.

 % \item[(c)] the greatest common divisor of all {\scriptsize  $ \Omega$}$_{ii}$ is equal to 1.
\end{itemize}
Let $t$ be a new parameter.
We specialize  the families  $\mbf{q, s\ {\rm and}\ t}$  of parameters as follows.
\begin{equation}\label{eq38}
  q_i=v_i,\quad s_{ij}=t^{-\Omega_{ij}}\quad {\rm and}\quad t_{ij}=t^{\Omega_{ji}}, \quad \forall i, j\in I.
\end{equation}
In the table below, we list the objects before the  specialization in (\ref{eq38})  at the first row and right below each object
 the resulting objects  after the specialization at  the second row.
% notations for the various objects obtained after the specializations at (\ref{eq38}).
\begin{center}
\renewcommand\arraystretch{1.6}
\begin{tabular}{|p{15pt}|p{20pt}|p{15pt}|p{15pt}|p{40pt}|p{80pt}|p{110pt}|}
 \hline
$\dot{\mscr{U}}$ & ${}_{\mscr A}\!\dot{\mscr{U}}$ &$\mscr{U}$ & $\mscr C$ &  $\mbb K$ & $\mscr A$  &  $\msf c_{i,\lambda}$ \\
\hline
$\dot{\mathfrak{U}}$& ${}_{\mrk Z}\!\dot{\mrk{U}}$ &$\mathfrak{U}$ &$\mrk C$& $\mbb Q(v,t)$ & $\mrk Z=\mbb{Z}[v^{\pm 1}, t^{\pm 1}]$ & $\sum_{j\in I}\lambda(j)$({\scriptsize $\Omega_{ij}-\Omega_{ji}$})\\
\hline
\end{tabular}
\end{center}
In particular,  $\mathfrak{U}$ is the two-parameter quantum algebra associated to $\Omega$ in \cite{FL12},  $\dot{\mathfrak{U}}$ its modified form,
and $\mathfrak{C}$  the category of weight modules of $\mathfrak{U}$. By Theorem ~\ref{thm8}, we have

\begin{cor}\label{cor1}
The statements in Theorem \ref{thm8} hold if   $\dot{\mscr{U}}$ (resp. $\mbb K$, $\mscr A$ and $\mscr C$) is replaced by  its specialization
$\dot{\mathfrak{U}}$ (resp. $\mbb Q(v,t)$, $\mrk Z$ and $\mathfrak{C}$).
\end{cor}

Notice that under the isomorphism defined by (\ref{eq36}) after specialization, the Verma and simple modules of $\U$  get sent to the respective modules of 
$\mathfrak{U}$. This implies that the decomposition numbers for simple modules in a Verma module of $\mathfrak U$ can be described by the Kazhdan-Lusztig polynomials (specialized at $q=1$) in  exactly the same manner as in  the setting  of the ordinary quantum algebras. 
The phenomenon applies to the rest of the cases when an identification of modified forms is established.

\subsection{Multi-parameter case}

Let $(q_{ij})_{i,j\in I}$ be a family of  indeterminates such that $q_{ij}q_{ji}=q_{ii}^{a_{ij}}$.
In this subsection, we specialize the families $\mbf{q, s}$ and $\mbf t$ as follows.
\begin{equation}\label{eq39}
  q_i=q_{ii}^{1/2},\quad s_{ij}=q^{1/2}_{ji}\quad {\rm and }\quad t_{ij}=q^{-1/2}_{ij},\quad \forall i, j\in I.
\end{equation}
Similarly, we use the following table to indicate the resulting objects after the specialization of $(\mbf{q, s, t})$ at (\ref{eq39}).
\begin{center}
\renewcommand\arraystretch{1.6}
\begin{tabular}{|p{15pt}|p{20pt}|p{15pt}|p{15pt}|p{100pt}|p{115pt}|p{90pt}|}
 \hline
$\dot{\mscr{U}}$ &${}_{\mscr A}\!\dot{\mscr{U}}$ &$\mscr{U}$ &$\mscr C$ & $\mbb K$ & $\mscr A$  &  $\msf c_{i,\lambda}$ \\
\hline
$\dot{\mcal{U}}$ &${}_{\mcal A}\!\dot{\mcal{U}}$& $\mcal{U}$ &$\mcal C$& $\mbb F=\mbb Q(q_{ij}^{1/2}, q_{ii}^{1/(2d_i)})$
 & $\mcal{A}=\mbb{Z}[q_{ij}^{\pm 1/2}, q_{ii}^{\pm 1/(2d_i)}]$ & $\prod_{j\in I}(q_{ij}q_{ji}^{-1})^{\lambda(j)/2}$\\
\hline
\end{tabular}
\end{center}
We note that $\mcal{U}$ and $\dot{\mcal{U}}$ are
 the multi-parameter quantum algebra  in \cite{HPR10} and its modified form, respectively, and
 $\mcal{C}$ is the category of weight modules of $\mcal{U}$.

% If we further assume that
%\begin{equation}\label{eq42}
%  q_{ii}^{1/2}=v_i,
%\end{equation}
%then we have the following corollary.

\begin{cor}\label{cor2}
Assume that $q_{ii}^{1/2}=v_i$ for any $ i\in I$.
Then the statements in Theorem \ref{thm8} hold if   $\dot{\mscr{U}}$ (resp. $\mbb K$, $\mscr A$ and $\mscr C$) is replaced by its specialization $\dot{\mathfrak{U}}$ (resp. $\mbb F$, $\mcal A$ and $\mathfrak{C}$).
\end{cor}

\subsection{Multi-parameter supercase, I}

We fix families of parameters  $\theta=\{\theta_{ij}\}_{i,j\in I}$ and $\mbf p=(\{p_{ij}\}_{i,j\in I}, \{p_i\}_{i\in I})$ such  that
\begin{equation}\label{eq24}
  p_{ij}^2=p_i^{2a_{ij}},\quad (p_{ij}p_{ji})/(\theta_{ij}\theta_{ji})=p_i^{2a_{ij}} \quad {\rm and}\quad  p_{ii}/\theta_{ii}=p_i^2,\  \forall i, j\in I.
\end{equation}
We further assume that
\begin{equation}\label{eq26}
p_i=v_i\gamma_i\ {\rm such\ that}\ \gamma_i^2=1,\quad \forall i\in I.
\end{equation}
For simplicity, we use the following  notations.
$$\tau_{ij}=p_{ij}p_i^{-a_{ij}},\quad \gamma_{ij}=\gamma_i^{a_{ij}}.$$
In this subsection, we shall specialize the families $\mbf{q, s}$ and $\mbf t$ in the following way.
\begin{equation}\label{eq40}
  q_i=p_i\gamma_i,\quad s_{ij}=\left\{\begin{array}{ll}
   \theta_{ij}^{-1}\gamma_{ij}& {\rm if}\ i\geq j,\vspace{3pt}\\
   \tau_{ji} & {\rm if}\ i<j,
 \end{array}\right.
 \quad
{\rm and }\quad t_{ij}=\left\{\begin{array}{ll}
   \theta_{ij}\tau_{ij}& {\rm if}\ i\geq j,\vspace{3pt}\\
   \tau_{ij}\tau_{ji}\gamma_{ij} & {\rm if}\ i<j,
 \end{array}\right.
 \quad \forall i, j\in I,
\end{equation}
where we fix a total order ``$<$'' on $I$.
Let $\mscr J$ be the two-sided ideal of $\mscr{U}$ generated by $K_iK_i'-1, \forall i\in I$.
The specialization of $\dot{\mscr{U}}$ (resp.  $\mscr C$, $\mbb K$, $\mscr A$ and $\msf c_{i,\lambda}$) at (\ref{eq40}) is as follows.
\begin{center}
\renewcommand\arraystretch{1.6}
\begin{tabular}{|p{15pt}|p{20pt}|p{28pt}|p{15pt}|p{115pt}|p{125pt}|p{73pt}|}
 \hline
 $\dot{\mscr{U}}$ &${}_{\mscr A}\!\dot{\mscr{U}}$&$\mscr{U}/\mscr J$ &$\mscr C$&  $\mbb K$ & $\mscr A$  &  $\msf c_{i,\lambda}$ \\
\hline
 $\dot{\msf{U}}_{\theta, \mbf p}$&${}_{\msf A}\!\dot{\msf{U}}_{\theta, \mbf p}$&$\msf{U}_{\theta, \mbf p}$ & $\msf C_{\theta,\mbf p}$ & $\mbb T=\mbb{Q}(\theta_{ij}, p_{ij}, p_{i}^{1/(2d_i)})$
 & $\msf A=\mbb{Z}[\theta_{ij}^{\pm 1}, p_{ij}^{\pm 1}, p_{i}^{\pm 1/(2d_i)} ]$ & $\prod_{j\in I}\tau_{ij}^{\lambda(j)}\gamma_{ij}^{\lambda(j)}$\\
\hline
\end{tabular}
\end{center}
We note that  $\msf c_{i,\lambda}^2=1$ after the specialization  under the assumptions (\ref{eq24}) and (\ref{eq26}).
Notice that  the quantum algebra $\msf U_{\theta, \mbf p}$ is the algebra  $U_{\theta, \mbf p}(\mrk g)$
 introduced in ~\cite{KKO13}.
From Theorem \ref{thm8}, we have

 \begin{cor}\label{cor3}
 Under the assumption (\ref{eq26}), we have  the following statements.
\begin{itemize}
\item[(a)] There is an isomorphism of algebras $\msf A\otimes_{\mbf A} {}_{\mbf A}\!\dot{\U}\simeq {}_{\msf A}\!\dot{\msf{U}}_{\theta, \mbf p}$.
 %the statements in Theorem \ref{thm8} hold if   $\dot{\mscr{U}}$ (resp. $\mbb K$, $\mscr A$ and $\mscr C$) is replaced by  $\dot{\msf{U}}_{\theta, p}$
%(resp. $\mbb T$, $\msf A$ and $\msf C_{\theta, p}$).

\item[(b)] The conjecture 2.8 in \cite{KKO13} holds.
\end{itemize}
\end{cor}

Note that Corollary \ref{cor3} implies that $\mbf C\simeq \msf C_{\theta,\mbf p}$, which has been proved in \cite{KKO13} for generic parameters.

\subsection{Multi-parameter supercase, II}

Consider the families $\tilde{\theta}=\{\tilde{\theta}_{ij}\}_{ i, j\in I}$ and $\tilde{p}=\{\tilde{p}_{i}\}_{i\in I}$ of parameters satisfying that
\begin{equation*}\label{eq27}
 \tilde{\theta}_{ij} \tilde{\theta}_{ji}= \tilde{p}_{i}^{-a_{ij}} \quad {\rm and}\quad
 \tilde{\theta}_{ii}= \tilde{p}_{i}^{-1},\quad \forall i, j\in I,
\end{equation*}
We further assume that
\begin{equation}\label{eq28}
\tilde{p}_i=v_i^2, \quad \forall i\in I.
\end{equation}
Let  $\zeta_{ij}=\tilde{\theta}_{ij}v_i^{a_{ij}}$.
In this subsection, the families $\mbf{q, s}$ and $\mbf t$ of parameters are specialized as follows.
\begin{equation}\label{eq41}
  q_i=\tilde p_{i}^{1/2},
  \quad s_{ij}=\left\{\begin{array}{ll}
   \zeta_{ji}v_{i}^{-a_{ij}} & {\rm if}\ i> j,\vspace{3pt}\\
   v_{i}^{-a_{ij}} & {\rm if}\ i\leq j,
 \end{array}\right.
 \quad
{\rm and }\quad t_{ij}=\left\{\begin{array}{ll}
   \zeta_{ij}& {\rm if}\ i\geq j,\vspace{3pt}\\
   1 & {\rm if}\ i<j,
 \end{array}\right.
 \quad \forall i, j\in I.
\end{equation}
The specialization of $\dot{\mscr{U}}$ (resp. $\mscr C$, $\mbb K$, $\mscr A$ and $\msf c_{i,\lambda}$) at (\ref{eq41}) is as follows.
\begin{center}
\renewcommand\arraystretch{1.6}
\begin{tabular}{|p{20pt}|p{25pt}|p{20pt}|p{105pt}|p{110pt}|p{20pt}|}
 \hline
 $\dot{\mscr{U}}$ & ${}_{\mscr A}\!\dot{\mscr{U}}$ &$\mscr C$& $\mbb K$ & $\mscr A$  &  $\msf c_{i,\lambda}$ \\
\hline
 $\dot{\msf{U}}_{\tilde \theta,\tilde p}$ &${}_{\widetilde{\msf A}}\!\dot{\msf{U}}_{\tilde \theta,\tilde p}$ &$\msf C_{\tilde \theta,\tilde p}$ &
 $\widetilde{\mbb T}=\mbb{Q}(\tilde \theta_{ij},\tilde p_{i}^{1/(2d_i)})$
 & $\widetilde{\msf A}=\mbb{Z}[\tilde \theta_{ij}^{\pm 1}, \tilde p_{i}^{\pm 1/(2d_i)} ]$ & $v_i^{\lambda_i}$\\
\hline
\end{tabular}
\end{center}
 The quantum algebra $\msf U_{\tilde \theta, \tilde p}$ is the algebra   $U_{\tilde \theta, \tilde p}(\mrk g)$ introduced in  \cite{KKO13}.
From Theorem \ref{thm8}, we have
\begin{cor}\label{cor4}
Under the assumption (\ref{eq28}), there is an isomorphism
$\widetilde{\msf A}\otimes_{\mbf A} {}_{\mbf A}\!\dot{\U}\simeq {}_{\widetilde{\msf A}}\!\dot{\msf{U}}_{\tilde \theta,\tilde p}$, as  algebras.
%the statements in Theorem \ref{thm8} hold if   $\dot{\mscr{U}}$ (resp. $\mbb K$, $\mscr A$ and $\mscr C$) is replaced by  $\dot{\msf{U}}_{\theta, p}$
%(resp. $\widetilde{\mbb T}$, $\widetilde{\msf A}$ and $\msf C_{\tilde \theta, \tilde p}$).
\end{cor}

\end{document}